\documentclass[11pt,leqno]{amsart}
\usepackage{amsmath,amssymb,amsthm}
\usepackage{hyperref}

\usepackage{bbm}

\DeclareMathOperator{\co}{co}

\DeclareMathOperator{\dens}{dens}
\DeclareMathOperator{\spann}{span}
\DeclareMathOperator{\card}{card}

\DeclareMathOperator{\kernel}{Ker}

\renewcommand{\geq}{\geqslant}
\renewcommand{\leq}{\leqslant}

\newtheorem{theorem}{Theorem}[section]
\newtheorem{lemma}[theorem]{Lemma}

\newtheorem{proposition}[theorem]{Proposition}

\theoremstyle{definition}

\newtheorem{example}[theorem]{Example}

\theoremstyle{remark}
\newtheorem{remark}[theorem]{Remark}
\numberwithin{equation}{section}

\def\fnote#1{\footnote}

\def\ignora#1{}
\def\n3#1{\left\vert  \! \left\vert \! \left\vert \, #1 \, \right\vert \!
  \right\vert \! \right\vert }

\newcommand{\iten}{\ensuremath{\widehat{\otimes}_\varepsilon}}
\newcommand{\pten}{\ensuremath{\widehat{\otimes}_\pi}}

\begin{document}



\title[$L$-orthogonality, octahedrality and Daugavet property]{$L$-orthogonality, octahedrality and Daugavet property in Banach spaces}

\author{Gin\'es L\'opez-P\'erez}\thanks{The research of Gin\'es L\'opez-P\'erez was supported by MICINN (Spain) Grant PGC2018-093794-B-I00 (MCIU, AEI, FEDER, UE), by Junta de Andaluc\'ia Grant A-FQM-484-UGR18
and by Junta de Andaluc\'ia Grant FQM-0185.}
\address[Gin\'es L\'opez-P\'erez]{Universidad de Granada, Facultad de Ciencias.
Departamento de An\'{a}lisis Matem\'{a}tico, 18071-Granada
(Spain)} \email{ glopezp@ugr.es}
\urladdr{\url{https://wpd.ugr.es/local/glopezp}}

\author{ Abraham Rueda Zoca }\thanks{The research of Abraham Rueda Zoca was supported by Vicerrectorado de Investigaci\'on y Transferencia de la Universidad de Granada in the program ``Contratos puente'', by MICINN (Spain) Grant PGC2018-093794-B-I00 (MCIU, AEI, FEDER, UE), by Junta de Andaluc\'ia Grant A-FQM-484-UGR18
and by Junta de Andaluc\'ia Grant FQM-0185.}
\address[A. Rueda Zoca]{Universidad de Granada, Facultad de Ciencias.
Departamento de An\'{a}lisis Matem\'{a}tico, 18071-Granada
(Spain)} \email{ abrahamrueda@ugr.es}
\urladdr{\url{https://arzenglish.wordpress.com}}

\maketitle 

\begin{abstract} 
In contrast with the separable case, we prove that the existence of almost $L$-orthogonal vectors in a nonseparable Banach space $X$ (octahedrality) does not imply the existence of nonzero vectors in $X^{**}$ being $L$-orthogonal to $X$, which shows that the answer to an environment question in \cite{gk} is negative. Furthermore, we prove that the abundance of almost $L$-orthogonal vectors in a Banach space $X$ (almost Daugavet property) whose density character is $\omega_1$ implies the abundance of nonzero vectors in $X^{**}$ being $L$-orthogonal to $X$. In fact, we get that a Banach space $X$ whose density character is $\omega_1$ verifies the Daugavet property if, and only if, the set of vectors in $X^{**}$ being $L$-orthogonal to $X$ is weak-star dense in $X^{**}$. We also prove that, under CH, the previous characterisation is false for Banach spaces with larger density character. Finally, some consequences on Daugavet property in the setting of $L$-embedded spaces are obtained.
\end{abstract}

\section{Introduction}

The concept of orthogonality in the setting of Banach spaces has been a central topic in the theory of Banach spaces. There are important and different concepts of orthogonality in Banach spaces in the literature as the given ones in \cite{ja} and \cite{ro}. For example, B. Maurey proved in \cite{maurey} that a separable Banach space contains an isomorphic copy of $\ell_1$ if and only if, there is a nonzero element $x^{**}\in X^{**}$ being symmetric orthogonal to $X$, in the terminology of \cite{ro}, that is, $\Vert x^{**}+x\Vert =\Vert x^{**}-x\Vert $ for every $x\in X$.  One of the strongest concepts of orthogonality is the $L$-orthogonality: two vectors $x$ and $y$ in a Banach space $X$ are called $L$-orthogonal if $\Vert x+y\Vert =\Vert x\Vert +\Vert y\Vert$. An element $x$ in $X$ will be called $L$-orthogonal to a subspace $Y$ of $X$ if $x$ is $L$-orthogonal to every element in $Y$. In the setting of Hilbert spaces, it is well known that for every closed and proper subspace there is a non-zero orthogonal vector to that subspace. In this sense, G. Godefroy proved in \cite[Theorem II.4]{god} that a separable Banach space $X$ containing isomorphic copies of $\ell_1$ can be equivalently renormed so that there is a vector $x^{**}$ in the unit sphere of $X^{**}$ being $L$-orthogonal to $X$. The aim of this note is to study the existence and abundance of vectors in the bidual space $X^{**}$ of a Banach space $X$ being $L$-orthogonal to $X$, in terms of the existence and abundance of vectors in $X$ which are almost $L$-orthogonal to finite-dimensional subspaces of $X$. It is natural to  say that a Banach space $X$ contains almost $L$-orthogonal vectors if, for every $x_1,\ldots ,x_n$ vectors in the unit sphere of $X$ and for every $\varepsilon>0$, there is some vector $x$ in the unit ball of $X$ such that $\Vert x+x_i\Vert>2-\varepsilon$ for every $1\leq i\leq n$. This is exactly equivalent to say that the norm of $X$ is octahedral, a concept considered by N. Kalton and G. Godefroy in \cite{gk}. In fact, it was proved in \cite{lola} that a Banach space $X$ containing isomorphic copies of $\ell_1$ can be equivalently renormed so that the new bidual norm is octahedral, and so a bidual renorming of $X^{**}$ contains almost $L$-orthogonal vectors. Notice that, in view of the characterisation given in \cite[Theorem 2.1]{blrjfa}, the previous result is related to the question whether a Banach space failing to be strongly regular can be equivalently renormed enjoying any diameter two property (see \cite{blradv} and references therein). Similarly, we will say that a Banach space $X$ has abundance of $L$- orthogonal vectors with respect to a norming subspace $Y$ of $X^*$ if, for every $x_1,\ldots ,x_n$ vectors in the unit sphere of $X$, for every nonempty $\sigma(X,Y)$-open subset $U$ of the unit ball of $X$ and for every $\varepsilon>0$, there is some vector $x$ in the unit ball of $X$ such that $\Vert x+x_i\Vert>2-\varepsilon$ for every $1\leq i\leq n$. This is exactly equivalent to say that $X$ satisfies the almost Daugavet property with respect $Y$ (see \cite{ksw1,ksw2} and Lemma \ref{lemma:tecniADP}). 

Recall that $X$ has the  \textit{Daugavet property with respect to $Y$} if every rank one operator $T:X\longrightarrow X$ of the form $T=y^*\otimes x$, for $x\in X$ and $y\in Y$, satisfies the equation
$$\Vert T+I\Vert=1+\Vert T\Vert,$$
where $I$ denotes the identity operator. If $Y$ is a norming of $X$, we say that $X$ has the \textit{almost Daugavet property}. We will say that $X$ has the Daugavet property if $Y=X^*$.

It is then natural to ask if for Banach spaces $X$ containing  or having abundance of almost $L$-orthogonal vectors one can find  some or many elements in $X^{**}$  being $L$-orthogonal to $X$. For example, in the case that $X$ is separable, G. Godefroy and N. Kalton proved in \cite[Lemma 9.1]{gk} that if $X$ contains almost $L$-orthogonal vectors, that is, the norm of $X$ is octahedral, then there are elements in $X^{**}$ being $L$-orthogonal to $X$, opening the question in the nonseparable setting.  

After some preliminary results in Section \ref{section:preliminares}, we prove in Section \ref{section:mainresuls} that the above question has a negative answer (Theorem \ref{octahedralnonorthogonal}), exhibiting examples of Banach spaces $X$ containing almost $L$-orthogonal vectors, that is, Banach spaces with an octahedral norm, whose bidual space lacks of nonzero vectors being $L$-orthogonal to $X$. In contrast with the above, we also prove in Section \ref{section:mainresuls} that the abundance of almost $L$-orthogonal vectors in a Banach space $X$ with $\dens(X)=\omega_1$ implies the abundance of vectors in $X^{**}$  being $L$-orthogonal to $X$ (Theorem \ref{theo:ADPpartial}). In other more precise words, if $X$ is a Banach space with the almost Daugavet property with respect to some norming subspace $Y$ of $X^*$ and $\dens(X)=\omega_1$, then the set of elements in $X^{**}$ being $L$-orthogonal to $X$ is 
$\sigma(X^{**},Y)$-dense in $X^{**}$. Then, as an immediate consequence, we get that a Banach space whose density character is $\omega_1$ satisfies de Daugavet property if, and only if, the set of elements in $X^{**}$ being $L$-orthogonal to $X$ is $w^*$-dense in $X^{**}$ (Theorem \ref{theocaradp}).  However, we show that the previous characterisation is false for higher density characters. Indeed, we even exhibit in Example \ref{example:w2false} an example of a Banach space $X$ with $\dens(X)=2^c$ enjoying the Daugavet property but having no $L$-orthogonal element, proving that Theorem \ref{theocaradp} is sharp.

In Section \ref{section:L-embebidos} we get some consequences on Daugavet property for Banach spaces being $L$-embedded. In particular we get that $X\pten Y$ has the Daugavet property, whenever $X$ is an $L$-embedded Banach space with $\dens(X)=\omega_1$ and $Y$ is a nonzero Banach space such that either $X^{**}$ or $Y$ has the metric approximation property (Theorem \ref{tensoLembesepa}).

\section{Preliminaries}\label{section:preliminares}

We will consider only real Banach spaces. Given a Banach space $X$, we will denote the unit ball and the unit sphere of $X$ by $B_X$ and $S_X$ respectively. Moreover, given $x\in X$ and $r>0$, we will denote $B(x,r)=x+rB_X=\{y\in X:\Vert x-y\Vert\leq r\}$. We will also denote by $X^*$ the topological dual of $X$. If $Y$ is a subspace of $X^*$, $\sigma(X,Y)$ will denote the coarsest topology on $X$ so that elements of $Y$ are continuous. Also, $Y$ is \textit{norming} if $\Vert x\Vert =\sup_{y\in Y,\Vert y\Vert\leq 1}\vert y(x)\vert$. Given a bounded subset $C$ of $X$, we will mean by a \textit{slice of $C$} a set of the following form
$$S(C,x^*,\alpha):=\{x\in C:x^*(x)>\sup x^*(C)-\alpha\}$$
where $x^*\in X^*$ and $\alpha>0$. If $X$ is a dual Banach space, the previous set will be called a \textit{$w^*$-slice} if $x^*$ belongs to the predual of $X$. Note that finite intersections of slices of $C$ (respectively of $w^*$-slices of $C$) form a basis for the inherited weak (respectively weak-star) topology of $C$. Throughout the text $\omega_1$ (respectively $\omega_2$) will denote the first uncountable ordinal (respectively the first ordinal whose cardinal is strictly bigger than the cardinality of $\omega_1$).

According to \cite{hww}, a Banach space $X$ is said to be an \textit{$L$-embedded space} if there exists a subspace $Z$ of $X^{**}$ such that $X^{**}=X\oplus_1 Z$. Examples of $L$-embedded Banach spaces are $L_1(\mu)$ spaces, preduals of von Neumann algebras, duals of $M$-embedded spaces or the dual of the disk algebra (see \cite[Example IV.1.1]{hww} for formal definitions and details). 

Given two Banach spaces $X$ and $Y$ we will denote by $L(X,Y)$ (respectively $K(X,Y)$) the space of all linear and bounded (respectively linear and compact) operators from $X$ to $Y$, and we will denote by $X\pten Y$ and $X\iten Y$ the projective and injective tensor product of $X$ and $Y$, respectively.  Moreover, we will say that $X$ has the \textit{metric approximation property} if there exists a net of finite rank and norm-one operators $S_\alpha:X\longrightarrow X$ such that $S_\alpha(x)\rightarrow x$ for all $x\in X$. See \cite{rya} for a detailed treatment of the tensor product theory and approximation properties.

Let $Z$ be a subspace of a Banach space $X$.
We say that $Z$ is an \emph{almost isometric ideal} (ai-ideal) in $X$ if
$X$ is locally complemented in $Z$ by almost isometries.
This means that for each $\varepsilon>0$ and for each
finite-dimensional subspace $E\subseteq X$ there exists a linear
operator $T:E\to Z$ satisfying
\begin{enumerate}
\item\label{item:ai-1}
  $T(e)=e$ for each $e\in E\cap Z$, and
\item\label{item:ai-2}
  $(1-\varepsilon) \Vert e \Vert \leq \Vert T(e)\Vert\leq
  (1+\varepsilon) \Vert e \Vert$
  for each $e\in E$,
\end{enumerate}
i.e. $T$ is a $(1+\varepsilon)$ isometry fixing the elements of $E$.
If the $T$ satisfies only (\ref{item:ai-1}) and the right-hand side of
(\ref{item:ai-2}) we get the well-known
concept of $Z$ being an \emph{ideal} in $X$ \cite{gks}.

Note that the Principle of Local Reflexivity means that $X$ is an ai-ideal in $X^{**}$
for every Banach space $X$. Moreover, there are well known Banach spaces properties, as the Daugavet property, octahedrality and
all of the diameter two properties, being inherited by ai-ideals
(see \cite{abrahamsen} and \cite{aln2}). Furthermore, given two Banach spaces $X$ and $Y$ and given an ideal $Z$ in $X$, then $Z\pten Y$ is a closed subspace of $X\pten Y$ (see e.g. \cite[Theorem 1]{rao}). It is also known that whenever $X^{**}$ or $Y$ has the metric approximation property then $X^{**}\pten Y$ is an isometric subspace of $(X\pten Y)^{**}$ (see \cite[Proposition 2.3]{llr2} and \cite[Theorem 1]{rao}). 

Throughout the text we will make use of the following two results, which we include here for the sake of completeness and for easy reference.

\begin{theorem}\label{theo:hbaioper}\cite[Theorem 1.4]{aln2}
Let $X$ be a Banach space and let $Z$ be an almost isometric ideal in $X$. Then there is a linear isometry $\varphi: Z^*\longrightarrow X^*$ such that
$$\varphi(z^*)(z)=z^*(z)$$
holds for every $z\in Z$ and $z^*\in Z^*$ and satisfying that, for every $\varepsilon>0$, every finite-dimensional subspace $E$ of $X$ and every finite-dimensional subspace $F$ of $Z^*$, we can find an operator $T:E\longrightarrow Z$ satisfying
\begin{enumerate}
\item $T(e)=e$ for every $e\in E\cap Z$, 
\item $(1-\varepsilon)\Vert e\Vert\leq \Vert T(e)\Vert\leq (1+\varepsilon)\Vert e\Vert$ holds for every $e\in E$, and;
\item $f(T(e))=\varphi(f)(e)$ holds for every $e\in E$ and every $f\in F$.
\end{enumerate}
\end{theorem}

Following the notation of \cite{abrahamsen}, to such an operator $\varphi$ we will refer as \textit{an almost-isometric Hahn-Banach extension operator}. Notice that if $\varphi:Z^*\longrightarrow X^*$ is an almost isometric Hahn-Banach extension operator, then $\varphi^*:X^{**}\longrightarrow Z^{**}$ is a norm-one projection (see e.g. \cite[Theorem 3.5]{kaltonlocomp}).

Another central result in our main theorems will be the following, coming from \cite[Theorem 1.5]{abrahamsen}

\begin{theorem}\label{theo:exteaiideales}
Let $X$ be a Banach space, let $Y$ be a separable subspace of $X$ and let $W\subseteq X^*$ be a separable subspace. Then there exists a separable almost isometric ideal $Z$ in $X$ containing $Y$ and an almost isometric Hahn-Banach extension operator $\varphi:Z^*\longrightarrow X^*$ such that $\varphi(Z^*)\supset W$.
\end{theorem}

According to \cite{gk}, given a Banach space $X$, the \textit{ball topology}, denoted by $b_X$, is defined as the coarsest topology on $X$ so that every closed ball is closed in $b_X$. As a consequence, a basis for the topology $b_X$ is formed by the sets of the following form
$$X\setminus\bigcup_{i=1}^n B(x_i,r_i),$$
where $x_1,\ldots, x_n$ are elements of $X$ and $r_1,\ldots, r_n$ are positive numbers.

Let us end by giving a pair of technical results which will be used in the proof of Theorem \ref{theo:almostdp}. The first one can be seen as a kind of generalisation of the classical Bourgain Lemma \cite[Lemma II.1]{ggms}, which asserts that, given a Banach space $X$, then every non-empty weakly open subset of $B_X$ contains a convex combination of slices of $B_X$. The following result already appeared in \cite{ksw2} without a complete proof. However, let us provide a proof here for the sake of completeness.

\begin{lemma}\label{lemabourrevisi}
Let $X$ be a Banach space and $Y\subseteq X^*$ be a norming subspace for $X$. Let $U$ be a non-empty $\sigma(X,Y)$ open subset of $B_X$. Then $U$ contains a convex combination of $\sigma(X,Y)$-slices of $B_X$.
\end{lemma}

\begin{proof} Let $\hat{U}$ be the $\sigma(X^{**},Y)$-open subset of $B_{X^{**}}$ defined by $U$. Notice that
$$B_{X^{**}}=\overline{\co}^{w^*}(\operatorname{Ext}(B_{X^{**}}))\subseteq \overline{\co}^{\sigma(X^{**},Y)}(\operatorname{Ext}(B_{X^{**}}))$$
by Krein-Milman theorem, so we can find a convex combination of extreme points $\sum_{i=1}^n \lambda_i e_i\in \hat U$. Since the sum in $X^{**}$ is $\sigma(X^{**},Y)$ continuous we can find, for every $i\in \{1,\ldots, n\}$, a $\sigma(X^{**},Y)$ open subset of $B_{X^{**}}$ such that $e_i\in V_i$ holds for every $i$ and such that $\sum_{i=1}^n \lambda_i V_i\subseteq \hat U$. Since the following chain of inclusions hold
$$\sum_{i=1}^n \lambda_i (V_i\cap B_X)\subseteq \left( \sum_{i=1}^n \lambda_i V_i\right) \cap B_X\subseteq \hat U\cap B_X=U,$$
the following claim finishes the proof.

\textbf{Claim:} Given $i\in\{1,\ldots, n\}$ we can find a slice $S_i$ such that $S_i\subseteq V_i\cap B_X$.

\begin{proof}[Proof of the Claim.]
By the definition of the $\sigma(X^{**},Y)$ we can assume that $V_i=\bigcap\limits_{j=1}^{k_i}T_j$ where every $T_j$ is a $\sigma(X^{**},Y)$-slice of $B_{X^{**}}$. Since $e_i\in V_i$ it follows that $e_i\notin \bigcup\limits_{j=1}^{k_i} B_{X^{**}}\setminus T_j$. Now $e_i$ is an extreme point of $B_{X^{**}}$ and then $e_i\notin \co\left(\bigcup\limits_{j=1}^{k_i} B_{X^{**}}\setminus T_j \right)$. Notice that $B_{X^{**}}\setminus T_j$ is $\sigma(X^{**},Y)$-closed in the $\sigma(X^{**},Y)$-compact space $B_{X^{**}}$ for every $j$ and, since it is additionally convex, it follows that $\co\left(\bigcup\limits_{j=1}^{k_i} B_{X^{**}}\setminus T_j \right)$ is $\sigma(X^{**},Y)$ compact too. Since $e_i\notin \co\left(\bigcup\limits_{j=1}^{k_i} B_{X^{**}}\setminus T_j \right)$ then we can find $x\in B_X$ such that $x\notin \overline{\co}^{\sigma(X,Y)}\left(\bigcup\limits_{j=1}^{k_i} B_{X^{**}}\setminus T_j \right)$. By a separation argument we can find $y^*\in S_Y$ and $\alpha>0$ such that
$$y^*(x)>\alpha>\sup\limits_{z\in Z} y^*(z)$$
for $Z=\bigcup\limits_{j=1}^{k_i} (B_{X^{**}}\setminus T_j)\cap B_X$. If we define
$$S_i:=\{z\in B_X: y^*(z)>\alpha\}$$
it follows that $S_i$ a $\sigma(X,Y)$-slice. Furthermore, given $z\in S_i$ it follows that $y^*(z)>\alpha$, so $z\in \bigcap\limits_{j=1}^{k_i} T_j\cap B_X=V_i\cap B_X$, which completes the proof of the claim.
\end{proof} \end{proof}

Let us end by giving a brief sketch of proof of the following lemma, which is an easy extension of \cite[Corollary 3.4]{ksw1}.

\begin{lemma}\label{lemma:tecniADP}
Let $X$ be a Banach space and assume that $X$ has the almost Daugavet property with respect to a norming subspace $Y\subseteq X^*$. Then, for every $x_1,\ldots, x_n\in S_X$, every $\varepsilon>0$ and every non-empty $\sigma(X,Y)$-open subset $U$ of $B_X$ there exists $z\in U$ such that
$$\Vert x_i+z\Vert>2-\varepsilon$$
for every $i\in\{1,\ldots, n\}$.
\end{lemma}

\begin{proof}
We will prove the lemma by induction on $n$. The case $n=1$ is just \cite[Corollary 3.4]{ksw1}.

Hence, assume by induction that the lemma holds for $n$, and let us prove it for $n+1$. To this end, pick $x_1,\ldots, x_{n+1}\in S_X$, $\varepsilon>0$ and $U$ to be a non-empty $\sigma(X,Y)$-open subset of $B_X$. By induction hypothesis we can find $z\in U$ such that
$$\Vert x_i+z\Vert>2-\frac{\varepsilon}{2}$$
holds for every $i\in\{1,\ldots, n\}$. For every $i\in\{1,\ldots, n\}$ choose $f_i\in S_{Y}$ such that $f_i(x_i+z)>2-\frac{\varepsilon}{2}$. Since $z\in U$ and $f_i\in Y$, it follows that 
$$z\in W:=U\cap \bigcap\limits_{i=1}^n S(B_X,f_i,\frac{\varepsilon}{2}).$$
Since $W$ is a non-empty $\sigma(X,Y)$ open subset of $B_X$ we can find $u\in W$ such that
$$\Vert x_{n+1}+u\Vert>2-\varepsilon.$$
Also, given $i\in\{1,\ldots, n\}$,  since $f_i(x_i)>1-\frac{\varepsilon}{2}$ and $f_i(u)>1-\frac{\varepsilon}{2}$ we get that
$$\Vert x_i+u\Vert\geq f_i(x_i+u)>2-\varepsilon,$$
which concludes the proof.\end{proof}

\section{Main results}\label{section:mainresuls}

Our first goal will be to show that, in contrast with the result in \cite[Lemma 9.1]{gk}, where it is proved that octahedrality of a separable Banach space $X$ is equivalent to the existence of elements in $X^{**}$ being $L$-orthogonal to $X$, this is no longer true in the nonseparable setting. That is, the existence of almost $L$-orthogonal vectors in a Banach space $X$, as defined in the introduction, does not imply the existence of nonzero vectors in $X^{**}$ being $L$-orthogonal  to $X$. For this, we need the following result.
 
\begin{proposition}\label{prop:previocontraeje}
Let $X$ be a uniformly smooth Banach space. Let $Y$ be a Banach space and assume that either $X^*$ or $Y^*$ has the approximation property and that there exists an element $T\in (X\iten Y)^{**}=(X^*\pten Y^*)=L(X^*,Y^{**})$ such that $\Vert T\Vert=1$ and such that
$$\Vert T+S\Vert=2$$
holds for every $S\in X\iten Y=K(X^*,Y)$. Then $T$ is a linear isometry.
\end{proposition}

\begin{proof}
Let $x^*\in S_{X^*}$ and $y\in S_Y$. Let us prove that $\Vert T(x^*)\Vert=1$. To this end, define $S:=x\otimes y\in X\iten Y$, where $x\in S_X$ satisfies that $x^*(x)=1$. By assumptions $\Vert T+S\Vert=2$. Consequently, for every $n\in\mathbb N$, there exists $x_n^*\in S_{X^*}$ such that
$$2-\frac{1}{n}<\Vert T(x_n^*)+S(x_n^*)\Vert\leq \Vert T(x_n^*)\Vert+\Vert S(x_n^*)\Vert=\Vert T(x_n^*)\Vert+\vert x_n^*(x)\vert.$$
The previous estimate implies that $\Vert T(x_n^*)\Vert\rightarrow 1$ and $\vert x_n^*(x)\vert\rightarrow 1$. This implies that, up taking a subsequence, either $x_n^*(x)\rightarrow 1$ or $x_n^*(x)\rightarrow -1$. Assume, up a change of sign, that $x_n^*(x)\rightarrow 1$. This implies that $\Vert x_n^*+x\Vert\rightarrow 2$ by evaluating the sequence at the point $x$. Then \cite[Fact 9.5]{checos} implies that $\Vert x_n^*-x^*\Vert\rightarrow 0$. Now, since $T(x_n^*)\rightarrow T(x^*)$ by continuity then $\Vert T(x_n^*)\Vert\rightarrow \Vert T(x^*)\Vert$. Consequently, $\Vert T(x^*)\Vert=1$. The arbitrariness of $x^*$ implies that $T$ is an isometry.\end{proof}

The previous lemma together with \cite[Theorem 3.2]{lr} yield a large class of counterexamples.

\begin{theorem}\label{octahedralnonorthogonal}
Let $X$ be a Banach space whose norm is octahedral. Let $I$ be an infinite set with $\card(I)>\dens(X^{**})$ and let $2<p<\infty$. Then the norm of $\ell_p(I)\iten X$ is octahedral but there is no $T\in (\ell_p(I)\iten X)^{**}$ such that $\Vert T\Vert=1$ and such that
$$\Vert T+S\Vert=1+\Vert S\Vert$$
 for every $S\in \ell_p(I)\iten X$.
\end{theorem}

\begin{proof}
Since $2<p<\infty$ it follows that $\ell_{q}(I)$ is finitely representable in $\ell_1$, where $\frac{1}{p}+\frac{1}{q}=1$, and has the MAP. By \cite[Theorem 3.2]{lr} it follows that the norm of $\ell_p(I)\iten X$ is octahedral. However, notice that there is no isometry $T:\ell_q(I)\longrightarrow X^{**}$ since $\dens(\ell_q(I))\geq \card(I)>\dens(\ell_1^{**})$: According to Proposition \ref{prop:previocontraeje}, there is no $T\in S_{(\ell_p(I)\iten X)^{**}}$ such that $\Vert T+S\Vert=2$ holds for every $S\in S_{\ell_p(I)\iten X}$, so we are done.
\end{proof}

Now, our goal will be to get nonzero vectors in the bidual of a Banach space $X$ being $L$-orthogonal to $X$ from the existence of almost $L$-orthogonal vectors in $X$. Let us show  another central result of the paper.

\begin{theorem}\label{theo:almostdp}
Let $X$ be a Banach space with the almost Daugavet property with respect to the norming subspace $Y\subseteq X^*$. Let $u\in B_{X^{**}}$. Then, for every almost isometric ideal $Z$ in $X$ with $\dens(Z)=\omega_1$ and for every $\{g_n:n\in\mathbb N\}\subseteq S_Y$ such that $g_n\in \varphi(Z^*)$ for every $n\in\mathbb N$,  we can find $v\in S_{X^{**}}$ satisfying the following two assertions:
\begin{enumerate}
\item $\Vert x+v\Vert=1+\Vert x\Vert$  for every $x\in Z$.
\item $v(g_n)=u(g_n)$  for every $n\in\mathbb N$.
\end{enumerate}
\end{theorem}

For the proof we will need the following lemma, which establishes the separable case.

\begin{lemma}\label{lemma:adpseparable}
Let $X$ be a Banach space with the almost Daugavet property with respect to the norming subspace $Y\subseteq X^*$. Let $u\in B_{X^{**}}$. Then, for every separable almost isometric ideal $Z$ in $X$ and for every $\{g_n:n\in\mathbb N\}\subseteq S_Y$ such that $g_n\in \varphi(Z^*)$ for every $n\in\mathbb N$,  we can find $v\in S_{X^{**}}$ satisfying the following two assertions:
\begin{enumerate}
\item $\Vert x+v\Vert=1+\Vert x\Vert$  for every $x\in Z$.
\item $v(g_n)=u(g_n)$  for every $n\in\mathbb N$.
\end{enumerate}
\end{lemma}

\begin{proof}
Let $\{g_n:n\in\mathbb N\}\subseteq S_Y$ and let $Z$ be a separable almost isometric ideal in $X$ and $\varphi:Z^*\longrightarrow X^*$ such that $\{g_n:n\in\mathbb N\}\subseteq \varphi(Z^*)$. Let us construct $v$. To this end, since $Z$ is separable, there exists a basis $\{O_n:n\in\mathbb N\}$ of the $b_Z$-topology restricted to $B_Z$. For every $n\in\mathbb N$ consider $\tilde{O_n}$ to be the $b_X$-open subset of $B_X$ which defines $O_n$ (i.e. if $O_n:=\bigcap\limits_{i=1}^{k_n} B_Z\setminus B(z^n_i,r_i)$ then $\tilde O_n:=\bigcap\limits_{i=1}^{k_n} B_X\setminus B(z^n_i,r_i)$). Since $X$ has the Daugavet property with respect to $Y$ it follows that, for every $n\in\mathbb N$, there exists by Lemma \ref{lemma:tecniADP} an element
$$x_n\in \bigcap\limits_{k=1}^n \tilde{O_k}\cap \bigcap\limits_{k=1}^n \left\{x\in B_X: \vert g_k(x)-u(g_k)\vert<\frac{1}{n} \right\}.$$
Now, for every $\delta>0$, there exists a $\delta$-isometry $$T:E:=\spann\{z_1,\ldots, z_{k_n}, x_n\}\longrightarrow Z$$ such that $T(z_i)=z_i$ and that
$g_k(T(v))=\varphi(g_k)(v)$
holds for every $v\in E$ and every $k\in\{1,\ldots, n\}$. Taking into account the property defining $x_n$ and the fact that $\delta$ can be taken as small as we wish we can ensure the existence of 
$$z_n\in\bigcap\limits_{k=1}^n O_k\cap \bigcap\limits_{k=1}^n \left\{z\in B_Z: \vert \varphi^{-1}(g_k)(z)-u(g_k)\vert<\frac{1}{n} \right\}.$$
Now \cite[Lemma 9.1]{gk} ensures the existence of a suitable $w^*$-cluster point $u\in S_{Z^{**}}$ of $\{z_n\}$ such that
$$\Vert z+u\Vert=1+\Vert z\Vert$$
holds for every $z\in Z$. If we take $v\in(\varphi^*)^{-1}(u)$ then we have that
$$\Vert x+v\Vert\geq \Vert \varphi^*(x+v)\Vert=\Vert x+u\Vert=1+\Vert x\Vert$$
holds for every $x\in Z$. Also, it is clear, by definition of the sequence $\{x_n\}$ and the fact that $u$ is a $w^*$-cluster point, that $v(g_k)=u(g_k)$ holds for every $k\in\mathbb N$.
\end{proof}

\begin{proof}[Proof of Theorem \ref{theo:almostdp}]

Let $Z$ be an almost isometric ideal in $X$ of density character equal to $\omega_1$ and let $\varphi:Z^*\longrightarrow X^*$ be a almost isometric Hahn-Banach extension operator such that $\{g_n:n\in\mathbb N\}\subseteq \varphi(Z^*)\cap S_Y$. In order to construct $v$, pick $\{x_\beta:\beta< \omega_1\}\subseteq S_X$ to be a dense subset of $S_Z$. Let us construct by transfinite induction on $\omega_0\leq \beta<\omega_1$ a family $\{(Z_\beta,\varphi_\beta, \{f_{\beta,\gamma}: \gamma<\beta\}, v_\beta): \beta<\omega_1\}$ satisfying the following assertions:
\begin{enumerate}
\item $Z_\beta$ is a separable almost isometric ideal in $X$ containing $\bigcup\limits_{\gamma<\beta}Z_\gamma \cup\{x_\beta\}$ and $\{x_n:n\in\mathbb N\}\subseteq Z_{\omega_0}$.

\item $\varphi_\beta: Z_\beta^*\longrightarrow X^*$ is an almost isometric Hahn-Banach operator such that $\{f_{\gamma,\delta}: \delta<\gamma<\beta\}\cup\{g_n:n\in\mathbb N\}\subseteq \varphi_\beta(Z_\beta^*)$.

\item $v_\beta\in S_{X^{**}}$ satisfies that
$$\Vert z+v_\beta\Vert=1+\Vert z\Vert$$
 for every $z\in Z_\beta$ and $\{f_{\beta,\gamma}: \gamma <\beta\}\subseteq S_Y$ is norming for $Z_\beta\oplus\mathbb Rv_\beta$.
\item For every $\delta<\gamma<\beta<\omega_1$ it follows
$$v_\beta(f_{\gamma,\delta})=v_\gamma(f_{\gamma,\delta}),$$
and
$$v_\beta(g_n)=u(g_n)$$
holds for every $n\in\mathbb N$.
\end{enumerate}
The construction of the family will be completed by transfinite induction on $\beta$. To this end, notice that the case $\beta=\omega_0$ follows from Lemma \ref{lemma:adpseparable}. So, assume that $(Z_\gamma,\varphi_\gamma, \{f_{\gamma,\delta}:\delta\in\gamma\}, v_\gamma)$ has already been constructed for every $\gamma<\beta$, and let us construct $(Z_\beta, \varphi_\beta, \{f_{\beta,\gamma}: \gamma\in\beta\}, v_\beta)$.  Pick $v$ to be a $w^*$-cluster point of the net $\{v_\gamma: \gamma<\beta\}$ (where the order in $[0,\beta[$ is the classical order). Notice that, by induction hypothesis, for every $\delta_0<\gamma_0<\gamma<\beta$ we have that
$$v_\gamma(f_{\gamma_0,\delta_0})=v_{\gamma_0}(f_{\gamma_0,\delta_0}).$$
 Then, since $v$ is a $w^*$-cluster point of $\{v_\gamma\}_{\gamma<\beta}$, we get that
\begin{equation}\label{ecuacondicadenna}
v(f_{\gamma_0,\delta_0})=v_{\gamma_0}(f_{\gamma_0,\delta_0}).
\end{equation}
Because of the same reason, given $n\in\mathbb N$, we obtain that

\begin{equation}\label{ecuaccondiiguau}
v(g_n)=v_{\omega_0}(g_n)=u(g_n).
\end{equation} 
Now notice that the set
$$\{f_{\gamma,\delta}: \delta<\gamma<\beta\}\cup\{g_n:n\in\mathbb N\}$$
is countable because $\beta$ is a countable ordinal. Also, $\bigcup\limits_{\gamma<\beta}
Z_\gamma$ is separable. Then, by \cite[Theorem 1.5]{abrahamsen} there exists an almost isometric ideal $Z_\beta$ in $X$ containing $\bigcup\limits_{\gamma<\beta} Z_\gamma \cup\{x_\beta\}$ and an almost isometric Hahn-Banach extension operator $\varphi_\beta: Z_\beta^*\longrightarrow X^*$ such that 
$$\varphi_\beta(Z_\beta^*)\supset \{f_{\gamma,\delta}: \delta<\gamma<\beta\}\cup\{g_n: n\in\mathbb N\}.$$
Let us construct $v_\beta$. To this end, since $Z_\beta$ is separable then Lemma \ref{lemma:adpseparable} applies for $v\in B_{X^{**}}$. Consequently, we can find $v_\beta\in S_{X^{**}}$ such that
\begin{enumerate}
\item $\Vert z+v_\beta\Vert=1+\Vert x\Vert$  for every $x\in Z_\beta$, and 
\item $v_\beta(f_{\gamma,\delta})=v(f_{\gamma,\delta})$  for $\delta<\gamma<\beta$, and $v_\beta(g_n)=v(g_n)$  for every $n\in\mathbb N$.
\end{enumerate}
Take $\{f_{\beta,\gamma}: \gamma<\beta\}\subseteq S_{Y}$ being norming for $Z_\beta\oplus\mathbb R v_\beta$. It follows as before that $\{(Z_\gamma,\varphi_\gamma, \{f_{\gamma,\delta}: \delta<\gamma\}, v_\gamma): \gamma\leq \beta\}$ satisfies our purposes.

Now consider $v$ to be a $w^*$-cluster point of $\{v_\beta\}_{\beta\in\omega_1}$. 
Let us prove that $v_\alpha$ satisfies the desired properties.

\begin{enumerate}
\item Let us prove that $v(g_n)=u(g_n)$ for every $n\in\mathbb N$. To this end pick $\varepsilon>0$, $n\in\mathbb N$, and find $\gamma>\omega_0$ so that $\vert (v-v_\gamma)(g_n)\vert<\varepsilon$. Since $v_\delta(g_n)=u(g_n)$ holds for every $\delta\geq \omega_0$ it follows that
$$\vert (v-u)(g_n)\vert=\vert (v-v_\gamma)(g_n)\vert<\varepsilon.$$
Since $\varepsilon>0$ was arbitrary we are done.
\item Given $x\in S_Z$ it follows that 
$$\Vert x+v\Vert=2.$$
To this end, pick $\varepsilon>0$. Since $\{x_\beta:\beta<\omega_1\}$ is dense in $S_Z$ find $\beta<\omega_1$ such that $\Vert x-x_\beta\Vert<\frac{\varepsilon}{3}$. Since $\Vert x_\beta+v_\beta\Vert=2$ find $\gamma<\beta$ such that
$$(z_\beta+v_\beta)(f_{\beta,\gamma})>2-\frac{\varepsilon}{3}.$$
Now, given any $\beta'>\beta$ we have that
$$(z_\beta+v_{\beta'})(f_{\beta,\gamma})=(z_\beta+v_\beta)(f_{\beta,\gamma})>2-\frac{\varepsilon}{3}.$$
Since $v$ is a $w^*$-cluster point of $\{v_\beta:\beta<\omega_1\}$ we obtain that
$$2-\frac{\varepsilon}{3}\leq (z_\beta+v)(f_{\beta,\gamma})\leq \Vert x_\beta+v\Vert\leq \Vert x+v\Vert+\frac{\varepsilon}{3},$$
so $\Vert x+v\Vert>2-\varepsilon$. Since $\varepsilon>0$ was arbitrary we also conclude that $\Vert x+v\Vert=2$. Finally, since $x\in S_Z$ was arbitrary, a convexity argument yields that
$$\Vert x+v\Vert=1+\Vert x\Vert$$
holds for every $x\in Z$.
\end{enumerate}
\end{proof}

Since every Banach space is trivially an almost isometric ideal in itself, the following result follows.

\begin{theorem}\label{theo:ADPpartial}
Let $X$ be a Banach space with the almost Daugavet property with respect to $Y\subseteq X^*$ such that $\dens(X)=\omega_1$. Let $u\in B_{X^{**}}$ and $\{g_n:n\in\mathbb N\}\subseteq S_Y$. Then we can find $v\in S_{X^{**}}$ satisfying the following two assertions:
\begin{enumerate}
\item $\Vert x+v\Vert=1+\Vert x\Vert$  for every $x\in X$.
\item $v(g_n)=u(g_n)$  for every $n\in\mathbb N$.
\end{enumerate}
\end{theorem}

As a consequence we obtain the following strengthening of the Daugavet property, which extends \cite[Theorem 3.2]{rueda}.

\begin{theorem}\label{theocaradp}
Let $X$ be a Banach space with $\dens(X)=\omega_1$. The following assertions are equivalent:
\begin{enumerate}
\item\label{teodaugabidu1} $X$ has the Daugavet property, that is, for every $x\in S_X$, every non-empty relatively weakly open subset $W$ of $B_X$ and every $\varepsilon>0$ there exists $y\in W$ such that $\Vert x+y\Vert>2-\varepsilon$.

\item\label{teodaugabidu2} For every non-empty relatively weakly-star open subset $W$ of $B_{X^{**}}$ there exists $v\in S_{X^{**}}\cap W$ such that 
$$\Vert x+v\Vert=1+\Vert x\Vert$$
holds for every $x\in X$.
\end{enumerate}
\end{theorem}

\begin{proof}
(2)$\Rightarrow$(1) is obvious. For the converse, take a non-empty weakly-star open set $W$ of $B_{X^{**}}$ and $u\in W\cap S_{X^{**}}$. With no loss of generality we can assume that $W=\bigcap\limits_{i=1}^n S(B_{X^{**}}, f_i,\alpha_i)$, for suitable $f_i\in X^*$ and $\alpha_i>0$. By Theorem \ref{theo:ADPpartial} we can find an element $v\in S_{X^{**}}$ such that
\begin{enumerate}
\item $\Vert x+v\Vert=1+\Vert x\Vert$ for every $x\in X$ and,
\item $v(f_i)=u(f_i)$ for every $i\in\{1,\ldots, n\}$.
\end{enumerate}
Now condition (2) above implies that $v\in W$ since $u\in W$, so we are done.\end{proof}

The following two examples show that Theorems \ref{theo:ADPpartial} and \ref{theocaradp} are sharp, at least, under continuum hypothesis.

\begin{example}\label{exam:cantinonufuncio}
Assume CH, and let $X:=\ell_2(\omega_1)\iten C([0,1])$. Notice that $\dens(X)=\dens(X^*)=\omega_1$ ($\dens(C([0,1])^*=c=\omega_1$). It follows that $X$ has the Daugavet property (c.f. e.g. \cite[P. 81]{werner}). Now take a dense set $\{x^*_\alpha:\alpha<\omega_1\}\subseteq X^*$ and $x\in X$. Notice that if $u\in X^{**}$ satisfies that $u(x^*_\alpha)=x(x^*_\alpha)$ holds for every $\alpha\in \omega_1$ then $u=x$. This implies that Theorem \ref{theo:ADPpartial} does not hold if we take an uncontably family of functionals $\{g_\alpha: \alpha\leq \dens(X)\}$.
\end{example}

\begin{example}\label{example:w2false}
Assume CH and let $X=\ell_2(\omega_2)\iten C([0,1])$. Notice that, as in the above example, $X$ enjoys the Daugavet property. However, $X$ does not have any $L$-orthogonal. To see it, assume by contradiction that there exists an $L$-orthogonal $T\in L(\ell_2(\omega_2),C([0,1])^{**})$. By Proposition \ref{prop:previocontraeje} then $T$ is an into isometry. Moreover, notice that $T$ is an adjoint operator (say $T=S^*$) because $\ell_2(\omega_2)$ is reflexive. Now $S:C([0,1])^*\longrightarrow \ell_2(\omega_2)$ is surjective because $S^*$ is an isometry. However, this is a contradiction because $\card(C([0,1])^*)=c=\omega_1<\omega_2\leq \card(\ell_2(\omega_2))$.
\end{example}

\section{Daugavet property and $L$-embedded spaces}\label{section:L-embebidos}

In order to obtain more consequences from Theorem \ref{theocaradp} we consider the following characterisation of the Daugavet property in $L$-embedded spaces, which is an extension of \cite[Theorem 3.3]{rueda}.

\begin{theorem}\label{equiLembe}
Let $X$ be an $L$-embedded Banach space with $\dens(X)=\omega_1$. Assume that $X^{**}=X\oplus_1 Z$. Then, the following are equivalent:
\begin{enumerate}
\item \label{equiLembe1} $X^*$ has the Daugavet property.
\item \label{equiLembe2} $X$ has the Daugavet property.
\item \label{equiLembe3} $B_Z$ is weak-star dense in $B_{X^{**}}$.
\end{enumerate}
\end{theorem}

\begin{proof}
(\ref{equiLembe1})$\Rightarrow$
(\ref{equiLembe2}) is obvious.

(\ref{equiLembe2})$\Rightarrow$
(\ref{equiLembe3}). Let $W$ be a non-empty weakly-star open subset of $B_{X^{**}}$ and let us prove that $B_Z\cap W\neq \emptyset$. By Theorem \ref{theocaradp} we can find $u\in W\cap S_{X^{**}}$ such that
$$\Vert x+u\Vert=1+\Vert x\Vert$$
for every $x\in X$. Since $u\in X^{**}$ we can find $x\in X$ and $z\in Z$ such that $u=x+z$. Now
$$1\geq \Vert z\Vert=\Vert -x+(x+z)\Vert=1+\Vert x\Vert.$$
This implies that $x=0$ and, consequently, $u\in B_Z$. So $W\cap B_Z\neq \emptyset$, as desired.

(\ref{equiLembe2})$\Rightarrow$
(\ref{equiLembe3}) follows from \cite[Theorem 2.2]{bm}.
\end{proof}

This result generalises \cite[Theorem 3.2]{bm}, where the authors proved that a real or complex $JBW^*$-triple $X$ has the Daugavet property if, and only if, its predual $X_*$ (which is an $L$-embedded Banach space) has the Daugavet property.

Now, following word-by-word the proof of \cite[Theorem 3.7]{rueda}, we get the next result, which gives an affirmative answer to \cite[Problem 5.2]{rueda} in the following sense.

\begin{theorem}\label{tensoLembesepa}
Let $X$ be an $L$-embedded Banach space with the Daugavet property and $\dens(X)=\omega_1$, and let $Y$ be a non-zero Banach space. If either $X^{**}$ or $Y$ has the metric approximation property then $X\pten Y$ has the Daugavet property.
\end{theorem}

\begin{proof} Assume with no loss of generality that $X^{**}=X\oplus_1 Z$. We will follow the ideas of \cite[Theorem 3.7]{rueda}. To this end, pick $G\in S_{L(X,Y^*)}$ and $\alpha>0$ and, to prove the theorem, it suffices to find an element $u\in S_{X^{**}}$ and $y\in S_Y$ such that $u(y\circ G)>1-\alpha$ and such that
$$\Vert z+u\otimes y\Vert_{(X\pten Y)^{**}}=1+\Vert z\Vert$$
 for every $z\in X\pten Y$. To do so, by the assumption that either $X^{**}$ or $Y$ has the MAP, it follows that $X^{**}\pten Y$ is an isometric subspace of $(X\pten Y)^{**}$ by \cite[Proposition 2.3]{llr2}, so it suffices to prove that
$$\Vert z+u\otimes y\Vert_{X^{**}\pten Y}=1+\Vert z\Vert$$
 for every $z\in X\pten Y$. To this end, find $x\in S_X$ and $y\in S_Y$ such that $G(x)(y)>1-\alpha$. This means that
$$x\in S(B_X, y\circ G,\alpha).$$
Since $S(B_{X^{**}}, y\circ G,\alpha)$ is a non-empty weakly-star open subset of $B_{X^{**}}$ and $X$ is an $L$-embedded Banach space with the Daugavet property then by Theorem \ref{equiLembe} we can find $u\in S_Z$ such that $u(y\circ G)>1-\alpha$. Let us prove that
$$\Vert z+u\otimes y\Vert_{X^{**}\pten Y}=1+\Vert z\Vert$$
 for every $z\in X\pten Y$. To this end pick $z\in X\pten Y$, $\varepsilon>0$, and take $T\in S_{L(X,Y^*)}$ such that $T(z)=\Vert z\Vert$. Since $\Vert u\Vert=1$ choose $x^*\in S_{X^*}$ such that $u(x^*)>1-\varepsilon$. Pick $y^*\in S_{Y^*}$ such that $y^*(y)=1$ and define $\hat T:X^{**}=X\oplus_1 Z\longrightarrow Y^*$ by the equation
$$\hat T(x+z)=T(x)+z(x^*)y^*.$$
It is not difficult to prove that $\Vert \hat T\Vert=1$. Hence
\[
\begin{split}
    \Vert z+u\otimes y\Vert_{X^{**}\pten Y}\geq \hat T(z+u\otimes y)=T(z)+u(x^*)y^*(y)& =1+u(x^*)\\
    & >2-\varepsilon.
\end{split}
\]
Since $\varepsilon>0$ was arbitrary we conclude the theorem.
\end{proof}

Let us end with some consequences about \textit{$u$-structure} in Banach spaces with the Daugavet property. To this end, according to \cite{gks}, given a Banach space $X$ and a subspace $Y$, we say that \textit{$Y$ is a $u$-summand in $X$} if there exists a subspace $Z$ of $X$ such that $X=Y\oplus Z$ and such that the projection $P:X\longrightarrow X$ such that $P(X)\subseteq Y$ satisfies that $\Vert I-2P\Vert\leq 1$ (in such a case we say that $P$ is a $u$-projection). We say that $Y$ is an \textit{$u$-ideal in $X$} if there exists a $u$-projection $P:X^*\longrightarrow Y^*$ such that $\kernel(P)=Y^\perp$, and we say that $Y$ is an \textit{strict $u$-ideal} in $X$ if $P(X^{***})$ is norming in $X^{***}$. Finally, we say that $X$ is an $u$-ideal if $X$ is an $u$-ideal in $X^{**}$ (under the canonical inclusion).

Let us end the section with the following two consequences of Theorem \ref{theocaradp} about $u$ structure in Banach spaces.

\begin{proposition}
Let $X$ be a Banach space with the Daugavet property and $\dens(X)=\omega_1$. Assume that $X$ is an $u$-summand in its bidual, say $X^{**}=X\oplus Z$. Then $Z$ is $w^*$-dense in $X^*$.
\end{proposition}

\begin{proof}
By Theorem \ref{theocaradp} it is enough to prove that every $u\in S_{X^{**}}$ such that
$$\Vert x+u\Vert=1+\Vert x\Vert$$
holds for every $x\in X$ satisfies that $u\in Z$. To this end, pick such an element $u\in S_{X^{**}}$. By the decomposition $X^{**}=X\oplus Z$ we get that there exist (unique) $x\in X$ and $z\in Z$ such that $u=x+z$. Let us prove that $x=0$. Notice that
$$1+2\Vert x\Vert=\Vert u-2x\Vert=\Vert u-2P(u)\Vert\leq \Vert I-2P\Vert\leq 1.$$
By the above inequality we obtain $\Vert x\Vert=0$ or, equivalently, that $u=z\in Z$, as we wanted.\end{proof}

\begin{remark}
In view of the previous proposition, we can wonder whether a Banach space $X$ with the Daugavet property can be a $u$-ideal in its bidual. The answer is positive (e.g. $L_1([0,1])$). However, as a consequence of \cite[Theorem 2.7]{ll}, the answer is negative if we require $X$ to be a strict $u$-ideal.
\end{remark}


\bigskip

\textbf{Acknowledgements:} The authors are grateful with D. Werner for pointing out mistakes and typos in an earlier version of the paper. They also thank G.~Godefroy and M.~Mart\'in for fruitful conversations.

\end{document}